\documentclass[11pt]{amsart}
\usepackage{stmaryrd}
\usepackage{mathrsfs}
\usepackage{amssymb}

\usepackage{titletoc}
\input xy
  \xyoption{all}
\usepackage{amscd}
\usepackage{amsmath, amssymb}
\usepackage{amsfonts}
\usepackage[colorlinks,linkcolor=blue,citecolor=blue, pdfstartview=FitH]{hyperref}

\numberwithin{equation}{section}

\theoremstyle{plain}
\newtheorem{thm}{Theorem}[section]
\newtheorem{theorem}[thm]{Theorem}
\newtheorem{lemma}[thm]{Lemma}
\newtheorem{corollary}[thm]{Corollary}
\newtheorem{proposition}[thm]{Proposition}

\theoremstyle{definition}

\newtheorem{question}[thm]{Question}

\newtheorem{remark}[thm]{Remark}

\newtheorem{definition}[thm]{Definition}

\newtheorem{example}[thm]{Example}

\newtheorem{defn-thm}[thm]{Definition-Theorem}
\newtheorem{conjecture}[thm]{Conjecture}

\newcommand{\R}{{\mathbb R}}

\newcommand{\qtq}[1]{\quad\mbox{#1}\quad}

\newcommand{\Om}{\Omega}

\newcommand{\btheorem}{\begin{theorem}}
\newcommand{\etheorem}{\end{theorem}}
\newcommand{\bproposition}{\begin{proposition}}
\newcommand{\eproposition}{\end{proposition}}
\newcommand{\bdefinition}{\begin{definition}}
\newcommand{\edefinition}{\end{definition}}
\newcommand{\bcorollary}{\begin{corollary}}
\newcommand{\ecorollary}{\end{corollary}}
\newcommand{\bproof}{\begin{proof}}
\newcommand{\eproof}{\end{proof}}
\newcommand{\bremark}{\begin{remark}}
\newcommand{\eremark}{\end{remark}}
\newcommand{\eexample}{\end{example}}
\newcommand{\bexample}{\begin{example}}
\newcommand{\la}{\langle}
\newcommand{\elemma}{\end{lemma}}
\newcommand{\blemma}{\begin{lemma}}
\newcommand{\ra}{\rangle}

\newcommand{\p}{\partial}

\renewcommand{\bar}{\overline}

\renewcommand{\phi}{\varphi}

\newcommand{\ee}{\end{eqnarray*}}
\newcommand{\be}{\begin{eqnarray*}}

\newcommand{\beq}{\begin{equation}}
\newcommand{\eeq}{\end{equation}}

\newcommand{\bd}{\begin{enumerate}}
\newcommand{\ed}{\end{enumerate}}

\renewcommand{\hat}{\widehat}
\renewcommand{\tilde}{\widetilde}

\usepackage{fancyhdr}
\renewcommand{\leftmark}{{\scriptsize On  Euler characteristic and fundamental groups of  compact  manifolds }}

\begin{document}
\title{On  Euler characteristic and fundamental groups of  compact  manifolds}
\makeatletter
\let\uppercasenonmath\@gobble
\let\MakeUppercase\relax
\let\scshape\relax
\makeatother
\author{Bing-Long Chen and Xiaokui Yang}

\address{{Address of Bing-Long Chen: Department of Mathematics,
Sun Yat-sen University, Guangzhou, P.R.China, 510275.}}
\email{\href{mailto:mcscbl@mail.sysu.edu.cn}{{mcscbl@mail.sysu.edu.cn}}}
\address{{Address of Xiaokui Yang: Morningside Center of Mathematics, Institute of
        Mathematics, HCMS, CEMS, NCNIS, HLM, UCAS, Academy of Mathematics and Systems Science, Chinese Academy of Sciences, Beijing 100190, China}}
\email{\href{mailto:xkyang@amss.ac.cn}{{xkyang@amss.ac.cn}}}
\maketitle

\begin{abstract} Let $M$ be a compact Riemannian manifold, $\pi:\tilde
M\>M$ be the universal covering and $\omega$  be a smooth $2$-form
on $M$ with $\pi^*\omega$ cohomologous to zero.  Suppose the
fundamental group $\pi_1(M)$ satisfies certain radial quadratic
(resp. linear) isoperimetric inequality, we show that there exists a
smooth $1$-form $\eta$ on $\tilde M$ of linear (resp. bounded)
growth  such that $\pi^*\omega=d \eta$.
 As applications, we prove that on a compact K\"ahler manifold $(M,\omega)$  with $\pi^*\omega$ cohomologous to zero,  if  $\pi_1(M)$  is
  $\mathrm{CAT}(0)$ or automatic (resp. hyperbolic), then $M$ is K\"ahler non-elliptic (resp. K\"ahler hyperbolic) and the Euler characteristic $(-1)^{\frac{\dim_\R M}{2}}
   \chi(M)\geq 0$ (resp. $>0$).

\end{abstract}

\setcounter{tocdepth}{1} \tableofcontents

\section{Introduction}
In differential geometry, there is a well-known conjecture due to H.
Hopf (e.g. \cite[Problem~10]{Yau82}):
\begin{conjecture} [Hopf] \label{hopf} Let $M$ be a compact, oriented and even dimensional Riemannian manifold of negative sectional curvature $K<0$.
 Then the signed Euler characteristic   $(-1)^{\frac{n}{2}}\chi(M)>0$, where $n$ is the real dimension of $M$.
\end{conjecture}

\noindent For $n=4$,  Conjecture \ref{hopf} was proven by S. S.
Chern (\cite{Ch55}) (who attributed it to J. W. Milnor).   Not much
has been known in higher dimensions. This conjecture can not be
established just by use of the Gauss-Bonnet-Chern formula (see
\cite{Ger76,Kle76}).  I. M. Singer suggested that in view of the
$L^2$-index theorem  an
 appropriate vanishing theorem for $L^2$-harmonic forms on the universal
 covering of $M$ would imply the conjecture (e.g. \cite{Dod79}).   In
the work \cite{Gro91}, Gromov introduced the notion of K\"ahler
hyperbolicity for  K\"ahler manifolds which means the K\"ahler form
on the universal cover is the exterior differential of some  bounded
$1$-form.  He established that the $L^2$-cohomology groups of the
universal covering of a K\"ahler hyperbolic manifold are not
vanishing only in the middle dimension. Combining this result and
the covering index theorem of Atiyah,  Gromov showed
$(-1)^{\frac{n}{2}}\chi(M)>0$ for a K\"ahler hyperbolic manifold
$M$. One can also show that a compact K\"ahler manifold homotopic to
a compact Riemannian manifold of negative sectional curvature is
K\"ahler hyperbolic and the canonical bundle is ample(\cite{CY17}).

When the sectional curvature of the manifold  is non-positive, it is
natural to consider whether  $(-1)^{\frac{n}{2}}\chi(M)\geq 0$
holds.    It should be noted that when the sectional curvature of a
compact K\"ahler manifold  is nonpositive, \cite{JZ00} and
\cite{CX01} proved independently  that the vanishing theorem of
Gromov type still holds and the Euler characteristic satisfies
$(-1)^{\frac{n}{2}}\chi(M)\geq 0$.   Actually, they proved that the
results also hold if the pulled-back K\"ahler form on the universal
covering is the exterior differential of some  $1$-form with linear
growth, and such manifolds are called  K\"ahler non-elliptic
(\cite{JZ00}) and K\"ahler parabolic (\cite{CX01}) respectively.  In
the sequel, for simplicity, we shall use one of these notions, e.g.,
K\"ahler non-elliptic. Moreover, Jost and Zuo proposed an
interesting  question in \cite[p.~4]{JZ00} that whether there are
some topological conditions to ensure the manifolds to be  K\"ahler
non-elliptic. One can also propose the following

\begin{question} \label{que}Let $M$ be a compact Riemannian manifold, $\pi: \hat{M}\rightarrow M$  a Galois  covering and $G$  the  group of covering transformations.
Let $\omega$ be a closed $q$-form ($q\geq 2$) on $M$ such that
$[\pi^{\ast} \omega]=0$ in $H^{q}_{\mathrm{dR}}(\hat{M})$. Find a
$(q-1)$-form $\eta$ on $\hat{M}$ of least growth order (in terms of
the distance function on $\hat{M}$) such that
$\pi^{\ast}\omega=d\eta$.
\end{question}

\noindent It is clear that the growth order of $\eta$ does not
depend on the choices of the metrics on $M$, and it should depend on
the geometry of the covering  transformation group $G$. Recall that
by a theorem of Gromov, a discrete group $G$ is
 hyperbolic if and only if it satisfies a linear isoperimetric inequality.
An expected answer for Question \ref{que} would be a relation
between certain isoperimetric inequality of the covering
transformation group $G$ and the least growth order of $\eta$.

We need some well-known notions of discrete groups to formulate an
isoperimetric inequality in our setting. Suppose $G=\la S|R\ra$ is a
finitely presented group, where $S$ is a finite symmetric set
generating $G$, $S=S^{-1}$, and $R$ is a finite  set (relator set)
in the free group $F_S$ generated by $S$.  The word metric on $G$
with respect to $S$ is defined as \beq  d_{S}(a,b)=\min\{n:
b^{-1}a=s_1s_2\cdots s_n, s_i\in S\}. \label{wm}\eeq For a word
$w=s_1s_2\cdots s_n$, its length $L(w)$ is defined to be $n$. If the
word $w=s_1\cdots s_n \in F_S$ representing the identity $e$ in $G$,
there are reduced words $v_1,\cdots, v_k$ on $S$ such that \beq
w=\prod_{i=1}^k v_ir_iv_i^{-1},\ \ \ r_i\ \ \text{\emph{or}}\ \
r_i^{-1}\in R \label{length2}\eeq as elements in $F_S$. The
\emph{combinatorial area} $\text{Area}(w)$ of $w$ is the smallest
possible $k$ for equation (\ref{length2}).

\begin{definition}\label{isop}We say a finitely presented  group $G=\la S|R\ra$ satisfies
 a \emph{radial isoperimetric inequality} of degree $p$, if there is a  constant $C>0$  such that
   for any word $w=s_1\cdots
s_n\in F_S$ of length $L(w)=n$ representing the identity $e$ in $G$,
we have
\begin{equation} \label{pii}
 \text{Area}(w)\leq C\sum_{i=1}^{n}\left(d_{S}(\bar{w}(i),
 e)+1\right)^{p-1},
 \end{equation}
where $w(i)=s_1\cdots s_i$ is the $i$-th subword of $w$ and
$\bar{w}(i)\in G$ is the natural image (from $F_S$ to $G$) of the
word $w(i)$.

 \end{definition}

\noindent  It is easy to see that the above definition is
independent of the choice of the generating set $S$. For  $p=1$,
this definition is the same as the usual linear isoperimetric
inequality
$$
\text{Area}(w)\leq C L(w).
$$
For  $p>1$, Definition \ref{isop} is stronger than the usual
isoperimetric inequality. Actually, the radial isoperimetric
inequality (\ref{pii}) can imply
\begin{equation} \label{uii}
 \text{Area}(w)\leq C  (\mathrm{diam}(w)+1)^{p-1} L(w)  \leq CL(w)^{p}.
 \end{equation}
 We obtain in Theorem \ref{t3.1}  a complete answer to Question
 \ref{que} for $q=2$. For simplicity, we only formulate the polynomial growth
 case which  has many important applications.

 \begin{thm} \label{t2} Let $M$ be a compact Riemannian manifold, $\pi: \hat{M} \rightarrow M$
   a Galois  covering with covering transformation group $G$ and $H^1_{\mathrm{dR}}(\hat{M})=0$.
    Let $\omega$ be a closed  $2$-form on $M$ such that $[\pi^{\ast} \omega]=0\in H^2_{\mathrm{dR}}(\hat{M})$.
Assume that the group $G$ satisfies the radial  isoperimetric
inequality (\ref{pii}) of degree $p\geq 1$.
     Then  there exists  a smooth $1$-form $\eta$ on $\hat{M}$  such
    that  $\pi^{\ast}\omega=d\eta$ and
 \beq |\eta|(x)\leq C(d_{\hat{M}}(x,x_0)+1)^{p-1} \label{hh}
 \eeq for all   $x\in \hat{M}$
 where $C$ is a positive constant and  $x_0\in \hat{M}$ is a fixed point.
 \end{thm}

\noindent Let's explain briefly the key ingredients in the proof of
Theorem \ref{t2} and demonstrate the significance of the radial
isoperimetric inequality (\ref{isop}). Let  $G=\la S|R\ra$ be the
finitely presented covering transformation group. The condition
$[\pi^*\omega]=0$ implies $\pi^*\omega=d\eta$ for some smooth
$1$-form $\eta$ on $\tilde M$. \bd\item[(i)] We show that there
exists a constant $C$ such that
  for any closed curve $\alpha$ on $\tilde{M}$, we can construct a word
  $w\in F_S$ representing the identity and approximating $\alpha$ such that \beq
\left|\int_{\alpha} \eta\right| \leq
C(L(\alpha)+\mathrm{Area}(w)).\eeq

\item[(ii)]  By using the radial isoperimetric inequality
(\ref{isop}), we prove \beq \int_\alpha \eta \leq C \int_\alpha
(d_{\hat{M}}(x,x_0)+1)^{p-1}ds\label{get}\eeq for any closed curve
$\alpha$.

\item[(iii)] $\eta$ can be regarded  as a ``bounded linear functional"
$L$ on the space of closed curves with a suitably defined norm.
Then we use the Hahn-Banach theorem and Whitney's local flat norm
(\cite{Wh57}) to find another bounded linear functional $\tilde L$
whose restriction on the space of closed curves is $L$. Moreover,
$\tilde L$ is represented by a differential form $\tilde \eta$ with
measurable coefficients with $\pi^*\omega=d\tilde \eta$ (in the
current sense) and \beq |\tilde\eta|(x)\leq
C(d_{\hat{M}}(x,x_0)+1)^{p-1},\ \ \qtq {a.e.}
 \eeq

 \item[(iv)] We  use the heat
equation method  to deform   $\tilde\eta$ to a smooth one with the
desired bound in (\ref{hh}).\ed  The radial isoperimetric inequality
(\ref{isop}) is the key ingredient in step (ii), and it could not
work for usual isoperimetric inequality for degree $p>1$.  With the
inequality (\ref{get}) in hand, step (iii) is classical (e.g.
\cite{Gro81, Gro98, Sik01, Sik05}). The smoothing process in step
(iv) is natural in the view point of PDE since $\omega=d\eta$ is
preserved under the specified heat equation (\ref{hhe})
and the  estimate (\ref{hh}) follows from standard apriori estimate of heat equations. \\

 As an application of Theorem \ref{t2}, we obtain
\begin{thm} \label{t3} Let $(M,\omega)$ be a compact K\"ahler manifold and $\pi: \tilde{M}\rightarrow M$ be the universal covering.
Suppose  $[\pi^{\ast} \omega]=0\in H_{\mathrm{dR}}^2(\tilde M)$. If
$\pi_{1}(M)$ satisfies the radial isoperimetric inequality
(\ref{pii}) of degree $p=2$, then $M$ is K\"ahler non-elliptic and
$(-1)^{\frac{\dim_\R M}{2}}\chi(M)\geq 0$.
\end{thm}

\noindent It is a natural question to ask whether  Theorem \ref{t3}
can hold for the usual quadratic isoperimetric inequality.
 It is well-known that  the automatic
groups and $\mathrm{CAT}(0)$ groups are important examples
satisfying the usual quadratic isoperimetric inequality. Recall that
$\mathrm{CAT}(0)$-groups are  groups which can act properly
discontinuously and cocompactly by isometries on  proper geodesic
$\mathrm{CAT}(0)$-metric spaces.
 Typical examples of $\mathrm{CAT}(0)$-groups are fundamental groups of compact manifolds with non-positive sectional curvature.
 Automatic groups   are finitely generated  groups  whose operations are governed by automata. For instances,  hyperbolic groups and mapping class groups are automatic (\cite{Mos95}).  These  groups  have
   been studied extensively in geometric group theory (e.g.
   \cite{CE92}).
 We prove in  Theorem \ref{aq} and Theorem \ref{bq} that
$\mathrm{CAT}(0)$ groups and automatic groups do satisfy the radial
isoperimetric inequality  (\ref{pii}) for $p=2$.  Hence, we obtain
the following result, which also gives an answer to the
aforementioned question proposed by Jost-Zuo:

\begin{thm} \label{t1} Let $(M,\omega)$ be a compact K\"ahler manifold and $\pi: \tilde{M}\rightarrow M$ be the universal covering.
Suppose  $[\pi^{\ast} \omega]=0\in H_{\mathrm{dR}}^2(\tilde M)$. If
$\pi_{1}(M)$ is $\mathrm{CAT}(0)$ or automatic, then $M$ is K\"ahler
non-elliptic and the Euler characteristic $(-1)^{\frac{\dim_\R
M}{2}}\chi(M)\geq 0$.
\end{thm}

\noindent One can also deduce the following result  easily from the
special case for $p=1$ in Theorem \ref{t2}:

\bcorollary\label{ks}  Let $(M,\omega)$ be a compact K\"ahler
manifold and $\pi: \tilde{M}\rightarrow M$ be the universal
covering. Suppose $[\pi^{\ast} \omega]=0\in H_{\mathrm{dR}}^2(\tilde
M)$.
  If $\pi_{1}(M)$ is a hyperbolic group, then $M$ is K\"ahler  hyperbolic
 and the Euler characteristic
 $(-1)^{\frac{\dim_\R M}{2}}\chi(M)>0$.

\ecorollary

\bremark  Note that the condition $[\pi^{\ast} \omega]=0$ is
equivalent to the fact that $\pi^*\omega=d\eta$ for some $1$-form
$\eta$, which holds trivially on K\"ahler hyperbolic or K\"ahler
non-elliptic manifolds.  It is not hard to see that this condition
 holds if $\pi_2(M)$ is torsion. In particular, one gets the
following fact pointed out by Gromov(\cite[p.~266]{Gro91}): if
$\pi_1(M)$ is hyperbolic and $\pi_2(M)=0$, then $M$ is K\"ahler
hyperbolic. \eremark

\bremark When $(M,\omega)$ is a symplectic manifold, one can obtain
the corresponding hyperbolicity or non-ellipticity from Theorem
\ref{t2} for $p=1$ or $2$ as in the K\"ahler case. See
\cite[Corollary~1.12]{Ked09} for the former case.

\eremark

\bremark It worths to point out that, every finitely generated
nilpotent group of class $p$ satisfies the usual isoperimetric
inequality of degree $p+1$ (\cite{GHY03}). On the other hand,
 Cheeger-Gromov proved in
\cite{CG86} that if the fundamental group $\pi_1(M)$ contains an
infinite normal amenable subgroup, then the Euler characteristic
$\chi(M)=0$. Hence it would be very interesting to ask: whether
$\pi_1(M)$ satisfies the radial isoperimetric inequality for $p=2$
if $\pi_1(M)$ does not  contain an infinite normal amenable
subgroup. If this were true, then any compact K\"ahler aspherical
manifold $M$ has non-negative signed Euler characteristic, i.e.
$(-1)^{\frac{\dim_\R M}{2}}
   \chi(M)\geq 0$. This is a special case of more general conjecture proposed by I.M. Singer :

    \emph{Conjecture: any compact  aspherical
manifold $M$ has non-negative signed Euler characteristic.}

\noindent When $X$ is an algebraic manifold, there is an algebraic
surface $Y$ such that $\pi_1(X)=\pi_1(Y)$ by  Lefschetz hyperplane
theorem. It is not hard to see that, in this case we only need to
consider the question for algebraic surfaces of general type thanks
to the Enriques-Kodaira classification (e.g. \cite[Chapter
V]{BHPV03}). \eremark

 In some sense,   the present paper is to build up one  relationship between
  the geometry of the fundamental groups $\pi_1(X)$ and the geometry$/$topology of the space $X$ itself.  There are many important works when the fundamental group $\pi_1(X)$ admits a finite dimensional representation with infinite image. We refer to \cite{Gro91, Mok92, Zuo94, ABCKT96, Eys99, Zuo99, JZ00,  BKT13,
 Kli13} and the references therein.\\

  The paper is organized as follows. In Section
\ref{hyperbolic},  to exhibit the key ingredients in the proof  of
Theorem \ref{t2}, we show a special case when the covering
transformation group $G$ is hyperbolic.  In Section \ref{general},
we deal with more general $G$ in Theorem \ref{t3.1} and establish
Theorem \ref{t2} and Theorem \ref{t3}.
 In Section \ref{ex}, we show that
$\mathrm{CAT}(0)$ groups and automatic groups satisfy the radial
isoperimetric inequality  (\ref{pii}) for $p=2$ and prove Theorem
\ref{t1}.

\vskip 1\baselineskip

\textbf{Acknowledgements.} The first author wishes to thank Hao
Liang for helpful discussions, and he was partially supported by
grants NSFC11521101, 11025107. The second author was partially
supported by China's Recruitment Program of Global Experts and  NSFC
11688101.

\section{Hyperbolic fundamental groups}\label{hyperbolic}

 Let $G$ be a finitely generated group  with a finite
  set $S$ generating $G$, where $S=S^{-1}$. The
\emph{Cayley graph} $\Gamma(G, S)$ of $G$ with respect to $S$ is a
graph satisfying \bd
\item The vertices of $\Gamma(G, S)$ are the elements of $G$;  \item
Two vertices $x,y \in G$ are joined by an edge if and only if there
exists an element $s\in S$ such that $x = ys$.\ed We define a metric
$d_{\Gamma(G,S)}$ on a Cayley graph $\Gamma(G,S)$ by letting the
length of every edge be $1$ and defining the distance between two
points to be the minimum length of arcs joining them. This metric
$d_{\Gamma(G,S)}$ is actually the word metric $d_S$ on $G$ defined
in (\ref{wm}). We say the group $G$ is hyperbolic if $(G,d_S)$ is a
hyperbolic metric space in the sense of Gromov(\cite{Gro87}, or
\cite[Proposition~2.6]{Oh}), i.e. there is a positive number
$\delta>0$ such that for any four points $x_1,x_2,x_3,x_4\in G$, we
have
$$
d_S(x_1,x_2)+d_S(x_3,x_4) \leq \max\{ d_S(x_1,x_3)+d_S(x_2,x_4);
d_S(x_1,x_4)+d_S(x_2,x_3)\} +2\delta.
$$
 One can check that the hyperbolicity of $G$ is independent of the
finite generating set $S$.  It can be shown that if the group $(G,
d_S)$ is hyperbolic, then $G$ is finitely presented, i.e.  $G=\la
S|R\ra$ where $R\subset F_S$ is a finite relator set. It is
well-known (\cite{Gro87}) that $G$ is hyperbolic if and only if it
satisfies the linear isoperimetric inequality:
\begin{equation} \label{li}
\text{Area}(w)\leq C L(w),
\end{equation}
for any word  $w$ in the free group $F_S$  representing the identity
$e$ in $G$. For more comprehensive introductions, we refer to
\cite{Gro87, Oh} and the references therein.

\btheorem\label{t2.1}  Let $M$ be a compact manifold and $\pi:
\tilde{M}\rightarrow M$ be a Galois covering such that
 $H^1_{\mathrm{dR}}(\tilde{M})=0$.
   Let $\omega$ be  a smooth closed $2$-form on $M$ such that $[\pi^{\ast}\omega]=0$ in
   $H^{2}_{\mathrm{dR}}(\tilde{M})$. If the covering transformation group $G$ is hyperbolic,
    then there exists a  bounded smooth $1$-form $\tilde{\eta}$ on $\tilde{M}$ such that $\pi^{\ast} \omega=d\tilde{\eta}$.
   \etheorem

  Fix a smooth reference metric on $M$ and endow $\tilde M$ with the induced metric. Since  $[\pi^{\ast}\omega]=0$ in $H^{2}_{\mathrm{dR}}(\tilde{M})$,  there is a smooth $1$-form ${\eta}$ on $\tilde{M}$
  such that $\pi^{\ast}\omega=d{\eta}$.  If ${\eta}$ is bounded, we are done. If it is not bounded,
  we shall show that there exists another bounded smooth $1$-form $\tilde{\eta}$ on $\tilde{M}$
   such that $\pi^{\ast}\omega=d\tilde{\eta}=d{\eta}$. The proof of Theorem \ref{t2.1} is
   divided into the following steps. \\

 \blemma\label{l2.2} For any $L>0$, there is a constant $D=D(L)$ depending on $L$ such that for
 any smooth closed curve $\alpha$ in $\tilde{M}$ with length $L(\alpha)\leq L$, we have
   \begin{equation}
\left|\int_{\alpha}{\eta}\right|\leq D . \label{estimate1}
 \end{equation} \elemma \bproof
 Let $\Omega\subset \tilde{M}$ be a fundamental domain of the group of deck transformations $G$.
 If $\alpha$ passes through $\Omega$, then the inequality (\ref{estimate1}) can be derived from  the smoothness of $\eta$.
 If $\alpha$ does not pass though $\Omega$, there exists some $g\in G$ such that $g^{-1}(\alpha)$ passes through $\Omega$.
 Hence,   $$\left|\int_{g^{-1}(\alpha)}\eta\right| \leq D$$ as in the previous case. Since $g^{\ast}(\omega)=\omega$,
  we know $$d(g^{\ast}\eta-\eta)=0.$$ On the other hand, since $H_{\mathrm{dR}}^1(\tilde{M})=0$, there exists a smooth function $f$
  such that $g^{\ast}\eta-\eta=df$.  Now (\ref{estimate1}) follows from
 $$
\left |\int_{\alpha}\eta\right|=
\left|\int_{g^{-1}(\alpha)}g^{\ast}\eta\right|=\left|\int_{g^{-1}(\alpha)}\eta+df\right|=\left|\int_{g^{-1}(\alpha)}\eta\right|\leq
D
 $$
 since $\alpha$ is closed and  $\int_{g^{-1}(\alpha)} df=0$.\eproof

\blemma\label{l2.3} There exists a constant $C>0$ such that for any
smooth closed curve $\alpha$ on $\tilde{M}$, we have
 \beq
\left |\int_{\alpha}{\eta}\right|\leq C\cdot
L(\alpha).\label{estimate2}
 \eeq

\elemma

\noindent Note that in the sequel the constant $C$ can vary from
line to line.

\bproof By perturbations, we may assume $\alpha: [0, L]\rightarrow
\tilde{M}$ is parameterized by the arclength, where $L=L(\alpha)$
and $|\alpha'(s)|\equiv 1$.
 Without loss of generality, we assume $L\geq 1$.
  Fix a base point  $x_0$ in the fundamental
domain $\Omega$. Let $[L]$ be the integer part of $L$. For each
$i=0,1,2\cdots, [L]$, we  choose $g_i\in G$ such that $\alpha(i)\in
g_i(\Omega)$. If $L$ is an integer, we choose $g_{L}=g_0$.
Otherwise, we set $g_{[L]+1}=g_0$. For simplicity, we set $N$ to be
$L$ if $L$ is an integer, otherwise $N$ is $[L]+1$. Connecting
$g_i(x_0)$ and $g_{i+1}(x_0)$ with a geodesic segment
$\tilde{\gamma}_i$, and $\tilde{\gamma}=\tilde{\gamma}_0\cup
\tilde{\gamma}_1 \cup \cdots \cup \tilde{\gamma}_{N-1} $ is a closed
curve of length
$$L(\tilde{\gamma}) \leq C L(\alpha)$$ for some $C$ independent of the curve
$\alpha$. We claim \begin{equation} \label{d1}
 \left|\int_{\alpha} \eta-\int_{\tilde{\gamma}} \eta\right| \leq C L(\alpha).
 \end{equation}
 Connecting $\alpha(i)$ and $g_i(x_0)$ with a geodesic segment $\beta_i$,
  we get $N$ closed curves $\hat{\gamma}_i=\alpha\mid_{[i,i+1]}\cup \beta_{i+1}\cup \tilde{\gamma}_i^{-1}\cup \beta_i^{-1}$,
    $i=0,1,\cdots, N-1$. Hence,
 \begin{equation}
  \int_{\alpha} \eta-\int_{\tilde{\gamma}} \eta =\sum_{i=0}^{N-1}\int_{\hat{\gamma}_i}
  \eta. \label{d2}
 \end{equation}
 By our construction, there is a constant $C$ such that $L(\hat{\gamma_i}) \leq C$.
  The claim (\ref{d1}) follows from the estimate (\ref{estimate1})   and (\ref{d2}).

 Let $G=\la S|R\ra$, where the generator set $S=S^{-1}$ and relator set $R$ are all finite.
 Let $\Gamma(G,S)$ be the Cayley graph of $G$.
  Then it is a standard fact that $\Gamma(G,S)$ is quasi-isometric to $\tilde{M}$, and this implies that there are  constants  $C_1,C_2>0$ such that
  \begin{equation}
  C^{-1}_1d_{\tilde{M}}(a(x_0),b(x_0))-C_2\leq d_{\Gamma(G,S)}(a,b)\leq C_1 d_{\tilde{M}}(a(x_0),b(x_0))+C_2
  \end{equation}
for any $a,b\in G$. From this,  we know
$d_{\Gamma(G,S)}(g_i,g_{i+1})$ are uniformly bounded by some
constant $C$. Hence,  for each $i$, $g_{i+1}=g_i s^{i}_1\cdots
s^{i}_{n_i}$ where $s^{i}_j\in S$, $n_i\leq C$. Connecting
$g_i(x_0)$, $(g_is^i_1)(x_0),$ $\cdots,$ $(g_is^{i}_1\cdots
s^{i}_{n_i})(x_0)$ with geodesic segments $\gamma^i_1$, $\cdots$,
$\gamma^i_{n_i}$ successively. It is clear the closed curve
$\tilde{\gamma}_i\cup(\gamma^i_1\cup \cdots\cup \gamma^i_{n_i}
)^{-1}$ has uniformly bounded length. By similar arguments as in
(\ref{d1}) and (\ref{d2}),  we conclude that \beq
  \left|\int_{\alpha} \eta-\int_{{\gamma}} \eta\right|  \leq C L, \label{d3}
 \eeq
where $\gamma=\left(\gamma^0_1\cup \cdots\cup \gamma^0_{n_0}\right)
\cup \cdots \cup \left(\gamma^{N-1}_1\cup \cdots\cup
\gamma^{N-1}_{n_{N-1}}\right)$. Consider  the  word $$w=s^0_1\cdots
s^{0}_{n_0} \cdots s^{N-1}_1\cdots s^{N-1}_{n_{N-1}}$$ representing
the identity in  $G$.  It has length \beq
L(w)=n_0+\cdots+n_{N-1}\leq CL. \label{wl}\eeq Let
$k_0=\mathrm{Area}(w)$ be the combinatorial area of $w$ and $w_0$ be
a word equal to $w$ as elements of $F_S$ with the form
 \begin{equation}\label{isow} w_0 =\prod_{i=1}^{ k_0} v_i r_iv_i^{-1} \end{equation}
 for reduced words $v_1,\cdots, v_{ k_0}$ on $S$ and $r_i$ or $r_i^{-1}\in
R$, $i=1,\cdots,  k_0$.   Note that  $w_0$ is obtained from $w$ by
inserting several subwords of the form $aa^{-1}$, where $a$ is a
word on $S$. According to  the previous construction  of the curve
$\gamma$ from  the word $w$, one can construct a curve $\gamma_0$
from the word $w_0$ by inserting several loops into $\gamma$ of the
form $\delta\delta^{-1}$ corresponding to the inserted subwords
$aa^{-1}$.  It is clear that
$$
\int_{{\gamma_0}} \eta= \int_{{\gamma}} \eta.
$$ Note that  each subword $v_i r_iv_i^{-1}$ in $w_0$  represents the identity in the group $G$.
In the correspondence  of $\gamma_0$ and the word $w_0$,  the
subword $v_i r_iv_i^{-1}$ represents a closed curve $\delta_i$ based
on $g_0(x_0)$, and $r_i$ represents a closed curve $\epsilon_i$
based on the end points of a curve $\eta_i$ which corresponds to
$v_i$. Then $\delta_i= \eta_i\cup \epsilon_i \cup \eta_i^{-1} $ and

$$
\int_{\delta_i} \eta=\int_{\epsilon_i} \eta.
$$
Since the relator set $R$ is finite, by Lemma \ref{l2.2}, there is a
constant $C>0$ such that $|\int_{\epsilon_i}\eta|\leq C$  and \beq
\left|\int_{\gamma} \eta\right|=\left|\sum_{i=1}^{k_0}
\int_{\delta_i} \eta\right| \leq C \mathrm{Area}(w). \label{d4}\eeq
Therefore, by (\ref{d3}) and (\ref{d4}), we obtain \beq
\left|\int_{\alpha} \eta\right| \leq
C(L(\alpha)+\mathrm{Area}(w)).\label{estimate21} \eeq From the
linear isoperimetric inequality (\ref{li}), there is a constant
$C_1>0$ such that \beq \mathrm{Area}(w)\leq C_1\cdot L(w) \leq
\tilde C\cdot L(\alpha), \eeq where the second inequality follows
from (\ref{wl}). By (\ref{estimate21}), we obtain the estimate
(\ref{estimate2}).\eproof

Note that, from (\ref{estimate2}), we can not deduce  the norm of
$\eta$   is bounded by $C$, since (\ref{estimate2}) holds only for
\emph{closed curves}. If we regard $\eta$ as a linear functional on
the space of general curves, $\eta$ is not necessarily bounded. We
shall use the Hahn-Banach theorem to find another bounded  $\tilde
\eta$ whose restriction on the space of closed curves is  $\eta$.
The following result is inspired by the classical results
\cite[Proposition~4.35]{Gro98} and \cite[Section~4.10]{Fed74} which
are based on  \cite{Wh57}.

\blemma\label{l2.4} There exists a $1$-form $\tilde \eta$ on $\tilde
M$ with bounded measurable coefficients such that
$\pi^*\omega=d\tilde \eta$ in the sense of currents. \elemma

\bproof We follow the framework of \cite[Chapter V and Theorem
~5A]{Wh57}. By the classical triangulation theorem (e.g.
\cite[Chapter~IV]{Wh57}), there is a simplicial complex $K$ and a
homomorphism $T:|K|\rightarrow \tilde{M}$ with the following
property. For each $n$-simplex $\sigma$ of $K$ ($n$ is the real
dimension of $\tilde{M}$), there is a coordinate system $\chi$ in
$M$ such that $\chi^{-1}$ is defined in a neighborhood of
$T(\sigma)$ in $\tilde{M}$ and $\chi^{-1} T$ is affine in $\sigma$.
Let $K^{(1)}$ be the first barycentric subdivision of $K$, and
$K^{(2)}$ be the second, $\cdots$.   For each $p>0$, consider the
chain groups (over $\mathbb{R}$), $C_{p}(K)$, $C_{p}(K^{(1)})$,
$C_{p}(K^{(2)})$, and etc. We regard a simplex $\sigma$ as the sum
of its subdivisions. Then there are natural inclusions
$C_{p}(K)\subset C_{p}(K^{(1)})\subset \cdots $. Denote
$C_p(K^{\infty})=\cup C_{p}(K^{(i)})$, and equip $C_p(K^{\infty})$ a
norm (mass) in the following manner: if $c=\sum a_i \sigma_i \in
C_{p}(K^{(j)})$ is a $p$-chain such that $\sigma_i$ are mutually
disjoint   (i.e. the interiors of $\sigma_i$ are mutually disjoint),
then
$$
\| c\|=\sum_{i}|a_i| \times \mathrm{Area}_{\tilde{M}}(T(\sigma_i)).
$$
Now $(C_p(K^{\infty}), \|\cdot \|)$ is a normed vector space. It can
be shown that for $c\in C_p(K^{\infty})$,
$$
\| c\| =\sup \left\{ \left|\int_{c} \theta \right|:  \theta \text{
is a $p$-form with} \  \|\theta\|\leq 1 \right\}.
$$
The norm (comass) of a differential form  $\theta$ of degree $p$ is
defined as
\begin{equation}
\|\theta\|=\sup_{\tilde M}\sup |\theta(e_1,\cdots,e_p)|
\end{equation}
where the inside supremum is taken  over all the orthonormal frames
$(e_1,\cdots, e_p,\cdots , e_n)$ in $T\tilde{M}$. See
\cite[p.~245--p.~246]{Gro98} for more details.  To invoke the main
result in \cite[Theorem~5A]{Wh57}, we need to recall the flat norm
$|\cdot|^{\flat}$ defined over $C_1(K^{\infty})$ by \beq
|c|^{\flat}=\inf \{\|c-\partial D\|+\|D\|\} \label{flat1} \eeq where
$D$ ranges over all $2$-chains. For any differential $1$-form
$\theta$, one can show \beq
|\theta|^{\flat}=\sup\left\{\left|\int_{c}\theta\right|:
|c|^{\flat}\leq 1\right\}=\max\left\{ \|\theta\|,
\|d\theta\|\right\}.\label{flat2} \eeq Let $Z_1(K^{\infty})= \cup
Z_1(K^{i})\subset C_1(K^{\infty})$ be the subspace of  all closed
$1$-chains (cycles).  We define a linear functional $\tilde{L}$ on
$Z_1(K^{\infty})$ in the following way: \beq \tilde{L}(c)=\int_{c}
\eta.\label{lf} \eeq According to Lemma \ref{l2.3}, we have
$|\tilde{L}(c)| \leq C\|c\|$ for any $c\in Z_1(K^{\infty})$.
Moreover,  by using
$$\int_{c}\eta=\int_{c-\partial D} \eta+\int_{D} \pi^*\omega, $$ we get
\beq |\tilde{L}(c)|\leq C (\|c-\partial D\|+\|D\|) \eeq for any
$2$-chain $D$ since $\omega$ is bounded.
 We conclude that  $$|\tilde{L}(c)| \leq C|c|^{\flat}$$ for any $c\in Z_1(K^{\infty})$.

By Hahn-Banach theorem, $\tilde{L}$ can be extended to a linear
functional  $L$ on $C_1(K^{\infty})$ satisfying \beq |L|^{\flat}=
|\tilde{L}|^{\flat}_{Z_1(K^{\infty})} \leq C.\eeq Since the flat
norm of the cochain $L$ is bounded, by \cite[Theorem~5A]{Wh57}, we
conclude that $L$ and $dL$ can be represented by differential forms
$\tilde{\eta}$ and $\tilde{\omega}$ with measurable coefficients,
i.e.
$$L(c)=\int_c \tilde \eta,\ \ \ (dL)(D)=\int_D \tilde \omega.$$
 According to the explanation in \cite[p.~260--p.~261]{Wh57}, the integration
$\int_c \tilde \eta$ is well-defined since $\tilde \eta$ is
measurable with respect to the $1$-dimensional  measure on $c$. For
the same reason, $\int_D \tilde \omega$ is also well-defined.
 Actually, \cite[Theorem~5A]{Wh57} is stated in a domain of the
Euclidean space, and we can apply the theorem  on each coordinate
chart and find a set of bounded differential forms coinciding on the
intersections of these coordinate charts, which give us global
bounded forms $\tilde{\eta}$ and $\tilde{\omega}$.  Note that
$$
\int_{D} \pi^*\omega=\int_{\partial D} \eta=\int_{\partial D}
\tilde{\eta}=L(\partial D)=(dL)(D)=\int_{D} \tilde{\omega}
$$
holds for any 2-chain $D$, and  it follows that
$\tilde{\omega}=\pi^*\omega$. Hence, we find a bounded $1$-form
$\tilde{\eta}$ such that $ \pi^*\omega=d\tilde{\eta}$ in the sense
of currents.\eproof

To complete the proof of Theorem \ref{t2.1},  we shall deform
$\tilde{\eta}$ in Lemma  \ref{l2.4} to a bounded smooth one so that
$\pi^*\omega=d\tilde \eta$ still holds. Indeed, such a form can be
obtained by using the heat equation method. Consider the following
heat equation for differential forms on $\tilde{M}$
\beq\begin{cases}
 \left(\frac{\partial}{\partial t}-\triangle\right) \tilde{\eta}(t)  = \pi^* d^{\ast}\omega,\\
 \tilde{\eta}(0) =\tilde{\eta}
\end{cases}\label{he}
\eeq
 where $-\triangle=(dd^{\ast}+d^{\ast} d)$ is the Hodge
Laplacian operator. Since $\tilde{M}$ has bounded geometry,  the
initial data $\tilde{\eta}$ is bounded, and $\pi^*d^{\ast} \omega$
is a bounded smooth form on $\tilde M$, it is well-known that the
equation (\ref{he}) admits a unique bounded short time solution
$\tilde{\eta}(t)$, $t\in [0, T_0]$, $T_0>0$. Moreover, $\tilde
\eta(t)$ is smooth when $t>0$. This can be done by solving the
equation on a sequence of bounded domains exhausting $ \tilde{M}$
and obtaining a priori estimates via applying the maximum principle.
Actually, the solution of (\ref{he})  exists for all time
$[0,\infty)$, but we do not need this. We claim
$d\tilde{\eta}(t)\equiv \omega$ for all $t\in [0,T_0]$  and so we
can take $\tilde{\eta}(T_0)$ as the desired bounded smooth form
$\tilde\eta$ in Theorem \ref{t2.1}.

By applying the standard Bernstein trick and the maximum principle
(localized version), one can show $|\nabla \tilde{\eta}(t)| \leq
\frac{C}{\sqrt{t}}$ for all $t\in (0,T_0]$. Taking exterior
differential of (\ref{he}), we get \beq\begin{cases}
 \left(\frac{\partial}{\partial t}-\triangle\right) \bar{\omega}(t)   =0,\\
  \bar{\omega}(t) \mid_{t=0} =0
\end{cases}\label{he1}
\eeq where $\bar{\omega}(t):= d\tilde{\eta}(t)-\pi^*\omega$. Let
$\hat{\omega}(t)=\int_{0}^{t} \bar{\omega}(s)ds$, then by
straightforward computations, we obtain \beq\begin{cases}
 \left(\frac{\partial}{\partial t}-\triangle\right) \hat{\omega}(t)   =0,\\
  \hat{\omega}(t) \mid_{t=0} =0.
\end{cases}\label{he2}\eeq
A simple calculation shows the estimate $|\nabla \tilde{\eta}(t)|
\leq \frac{C}{\sqrt{t}}$ implies that $\hat{\omega}(t) $ is
uniformly bounded on $\tilde{M}\times [0,T_0]$. By standard Bochner
formula, we obtain
$$ \left(\frac{\partial}{\partial t}-\triangle\right)\left(e^{-Ct}|\hat\omega|^2\right)\leq 0$$
for some large $C$ depending on the curvature bound of $\tilde M$.
 Since the maximum principle of the heat equation
on such manifolds  holds when the subsolution grows  not faster than
$Ce^{Cd(x,x_0)^2}$ (\cite[Theorem~15.2]{Li} or
\cite[Lemma~2.5]{Che10}), we obtain $\hat{\omega}(t)\equiv 0$, i.e.,
$d\tilde{\eta}(t) \equiv \pi^*\omega$. The proof of Theorem
\ref{t2.1} is completed. \qed

\vskip 2\baselineskip

 \section{General fundamental groups}\label{general}

 In this section, we  deal with a Galois covering $\pi: \tilde{M} \rightarrow M$ of  a  compact
 manifold $M$ with a finitely presented covering transformation   group $G=\la S|R\ra$.

\begin{thm}\label{t3.1}  Let $M$ be a compact manifold, $\pi: \tilde{M} \rightarrow M$ a Galois covering
with $H^{1}_{\mathrm{dR}}(\tilde{M})=0$. Let $\omega$ be  a smooth
closed $2$-form on $M$ such that  $[\pi^{\ast}\omega]=0$ in
$H^{2}_{\mathrm{dR}}(\tilde{M})$. Let  $f: \mathbb{R}_{+}\rightarrow
[1,+\infty)$  be  a  non-decreasing function. Suppose the covering
transformation group $G$ is  finitely presented and satisfies the
following radial isoperimetric inequality: for any word $w=s_1\cdots
s_m$ in the free group $F_S$ representing the identity $e$ in $G$,
we have
 \begin{equation} \label{if}
 \mathrm{Area}(w)\leq C\sum_{i=1}^mf(d_{S}(\bar{w}(i), e))
 \end{equation} where $w(i)=s_1\cdots s_i$
 and $\bar{w}(i)$ is its image in $G$. Then there exists a $1$-form
$\tilde{\eta}$ on $\tilde{M}$ with measurable coefficients
satisfying
\begin{equation} \label{if1}
\begin{cases}
 \pi^{\ast} \omega =d\tilde{\eta}\\
   |\tilde{\eta}|(x)  \leq C f(Cd(x,x_0))
 \end{cases}
 \end{equation}
 for some constant $C>0$ where $x_0$ is a fixed point in $\tilde M$.  Moreover,   if $f(\lambda)\leq Ce^{C\lambda^2}$, one  can  choose $\tilde \eta$ to be smooth.
 \end{thm}

 \begin{proof} We shall use a similar strategy as in the proof of  Theorem \ref{t2.1}.
 Let $\pi^{\ast}\omega=d\eta$ hold for some smooth $1$-form $\eta$ on $\tilde{M}$.  As in the proof of Lemma
 \ref{l2.3}, there exists a constant $C$ such that
  for any closed curve $\alpha$ on $\tilde{M}$, we can construct a word
  $w\in F_S$ representing the identity and approximating $\alpha$ such that
  \beq \label{step2} \left|\int_{\alpha} \eta\right| \leq
C(L(\alpha)+\mathrm{Area}(w)).\eeq Since $f\geq 1$ and it is
non-decreasing, by (\ref{if}), we have \be
\mathrm{Area}(w)&\leq& C\sum_{i}f(d_{S}(\bar{w}(i), e))\\
&\leq& C\sum_i f(Cd(x_0, \bar{w}(i)(x_0))+C)\\
&\leq& C\int_0^{L(\alpha)} f(C d(x_0, \alpha(s)))ds \ee where the
second inequality follows from the quasi-isometry between $G$ and
$\tilde M$, and the last inequality follows from the construction in
the proof of Lemma \ref{l2.3}. Note also that the constant $C$
varies from line to line. As in Lemma \ref{l2.4}, we define the same
linear functional $\tilde{L}$ on the subspace of closed $1$-chains
$Z_1(K^{\infty})$ by (\ref{lf}). However, $\tilde L$ is not
necessarily  a bounded linear functional in the original norm. We
shall  define a new norm $||\cdot||_{\check{g}}$ on
$C_1(K^{\infty})$ by choosing a conformal Riemannian metric
$\check{g}(x)=f(C d(x_0,x))^2
  g(x)$ on $\tilde M$. In terms of this new norm,
 \beq
 \|\alpha\|_{\check{g}}=\int_{0}^{L} f(Cd(x_0, \alpha(s)))ds \geq
 \frac{1}{C}\mathrm{Area}(w).\label{ii}
 \eeq
 Combining (\ref{step2}) and (\ref{ii}), we get
 \begin{equation} \label{step3}
\left|\int_{\alpha} \eta\right| \leq C  ||\alpha||_{\check{g}} .
\end{equation}
Hence,
\begin{equation} \label{nfn}
|\tilde{L}(c)|\leq C( \|c-\partial D\|_{\check{g}}+\|D\|_{g})
\end{equation}
for any $2$-chain $D$ and  closed $1$-chain $c$. Now define a new
flat norm $$|c|^{\flat}_{\check{g}}= \inf_{D} \|c-\partial
D\|_{\check{g}}+\|D\|_{g}$$ on $C_1(K^{\infty})$.  It is not
difficult to show that the dual flat norm
$|\theta|^{\flat}=\max\{\|\theta\|_{\check{g}}, \|d\theta\|_{g}\}$
for any 1-form $\theta$ on $\tilde{M}$.   We have shown
$\|\tilde{L}\|^{\flat}_{Z_1(K^{\infty})} \leq C$ from (\ref{nfn}).
As before, $\tilde{L}$ can be extended to a bounded linear
functional $L$ on the whole space $C_1(K^{\infty})$. By
\cite[Theorem~5A]{Wh57} again,  $L$ is represented by a differential
form $\tilde \eta$ with measurable coefficients such that
$\pi^*\omega=d\tilde\eta$ and $$\|\tilde \eta\|_{\check{g}}\leq C$$
which is equivalent to
$$|\tilde \eta|_g(x) \leq Cf(Cd(x,x_0)).$$
This completes the first part of Theorem \ref{t3.1}.

 When  $f(\lambda)\leq Ce^{C\lambda^2}$, we  use the heat
equation method again to deform  the $1$-form $\tilde\eta$ to a
smooth one. We first construct a solution $v(x,t)$ to the following
equation: \beq \begin{cases}\left(\frac{\p}{\p t}-\Delta\right)v=0\\
v|_{t=0}=d^*\tilde\eta.
\end{cases}\label{v}
\eeq It is well-known that  the heat kernel $H(x,y,t)$  (for
functions) on $\tilde{M}$ exists for all $t\geq 0$, and it satisfies
the Gaussian type estimates
\begin{equation}\label{het} \begin{cases}
H(x,y,t) \leq C t^{-\frac{n}{2}} e^{-\frac{d(x,y)^2}{Ct}}\\
|\nabla_{y} H|(x,y,t)  \leq Ct^{-\frac{n+1}{2}}
e^{-\frac{d(x,y)^2}{Ct}}
\end{cases}
\end{equation}
for all $t\in [0,1]$. Combining with (\ref{het}),  there exists a
small $T_0>0$ such  that for all $x\in \tilde{M}$, $0<t\leq T_0$,
the function \beq  v(x,t)=\int_{\tilde{M}} \langle
d_yH(x,y,t),\tilde \eta(y)\rangle dy \label{dv}\eeq solves (\ref{v})
and satisfies
\begin{equation}|v|(x,t)\leq Ct^{-\frac{1}{2}}
f(Cd(x,x_0))).\label{name}\end{equation}  Indeed,  for $d(x,x_0)>1$,
we have
\begin{equation}\label{eov}
\begin{split}
|v|(x,t) \leq & \int_{B(x, 2d(x,x_0))} Ct^{-\frac{n+1}{2}} e^{-\frac{d(x,y)^2}{Ct}}f(C(d(x_0,y))dy\\&
+\sum_{k=1}^{\infty} \int_{B(x, 2^{k+1}d(x,x_0))\setminus  B(x,2^{k}d(x,x_0))} Ct^{-\frac{n+1}{2}} e^{-\frac{d(x,y)^2}{Ct}}f(C(d(x_0,y))dy.\\
\end{split}
\end{equation}
The first term is dominated by
$$
C f(Cd(x_0,x)) t^{-\frac{1}{2}} \int_{B(x, 2d(x,x_0))}
t^{-\frac{n}{2}} e^{-\frac{d(x,y)^2}{Ct}}\leq C t^{-\frac{1}{2}}
f(Cd(x_0,x)).
$$
 Note also that  $f(C(d(x_0,y))) \leq Ce^{C4^kd(x_0,x)^2}$ on
$B(x, 2^{k+1}d(x,x_0))$ and
$$ {\mathrm{vol}(B(x, 2^{k+1}d(x,x_0)))}\leq e^{C2^{k+1}d(x,x_0)} $$ by volume comparison theorem since the Ricci curvature is bounded from below.
Hence, the second term of (\ref{eov}) can be  dominated by
$$
Cd(x_0,x)^{-(n+1)}{\left(\frac{d(x,x_0)}{\sqrt{t}}\right)}^{{n+1}}\sum_{k=1}^{\infty}
e^{-4^{k}\frac{d(x_0,x)^2}{Ct}} \cdot e^{C4^kd(x,x_0)^2} \leq C
d(x_0,x)^{-(n+1)}\leq C
$$
when $t<T_0$, where $T_0$ is a suitable small constant. If
$d(x,x_0)\leq 1$, we  have $d(x_0,y)\leq d(x,y)+1$, and \beq
|v|(x,t) \leq  \int_{\tilde M} Ct^{-\frac{n+1}{2}}
e^{-\frac{d(x,y)^2}{Ct}} e^{C d(x,y)^2}dy\leq C
t^{-\frac{1}{2}},\eeq which completes the proof of (\ref{name}).

Let $u(x,t)=\int_{0}^{t} v(x, t-s)ds$, then by (\ref{v}), it is easy
to verify that $u$ is a weak solution to the following
  heat equation  on $\tilde{M}$:
\begin{equation}
\begin{cases}
(\frac{\partial}{\partial t}-\triangle)u=d^{\ast} {\tilde \eta}\\
u\mid_{t=0}=0
\end{cases}
\end{equation}
and $u$ satisfies $$|u|(x,t)\leq C\sqrt{t} f(Cd(x,x_0)).$$ Multiply
both sides of the above equation with $u\xi$, where $\xi$ is a
standard cutoff function which is $1$ on $B(x,1)$ and $0$ outside of
$B(x,2)$. An integration by part shows
\begin{equation}
\int_{0}^{T_0} \int_{B(x,1)} |du|(y,t)^2dydt\leq C \int_{0}^{T_0}
\int_{B(x,2)} u^2(y,t)+|\tilde \eta|(y)^2 dydt
\end{equation}
which implies
\begin{equation} \label{ine}
\int_{0}^{T_0} \int_{B(x,1)} |du|(y,t)^2+|\tilde \eta|^2(y)dydt\leq
C f^2(C(d(x,x_0)).
\end{equation}
 It is easy to verify that $\tilde \eta(t)=\tilde\eta-du$
is a weak solution of  the following equation  on $\tilde{M}\times
(0,T_0]$:
\begin{equation}\label{hhe}\begin{cases}
 \left(\frac{\partial}{\partial t}-\triangle\right) {\tilde \eta}(t)  =\pi^*d^{\ast} \omega\\
 \tilde{\eta}(0)={\tilde \eta}.
\end{cases}
\end{equation}

\noindent Similar to the equation (\ref{he}) in the proof of Theorem
\ref{t2.1}, we know  $\tilde\eta(x,t)$ is a smooth solution on
$\tilde{M}\times (0,T_0]$. A standard computation shows
\begin{equation}\label{hhei}
 \left(\frac{\partial}{\partial t}-\triangle\right) |{\tilde\eta}(t)|^2  \leq -2|\nabla\tilde \eta|(t)^2 +C|\tilde\eta|(t)^2+C.
\end{equation}
Hence, we  obtain
 $$\left(\frac{\partial}{\partial t}-\triangle\right)\left(|{\tilde\eta}(t)|^2+1\right)e^{-Ct}  \leq
 0.$$
By the well-known result \cite[Theorem~1.2]{LT91} for non-negative
subsolution of heat equation, we obtain
$$
|\tilde \eta|^2(x,t) \leq 2 \sup_{y\in B(x,1)}|\tilde \eta|^2(y,0) +
C\int_{0}^{T_0} \int_{B(x,1)} |\tilde \eta|^2(y,t)dydt+C \leq
Cf^2(Cd(x,x_0)),
$$ where the second inequality follows from (\ref{ine}).
On the other hand,  a  local gradient estimate on $\tilde \eta$
implies
\begin{equation}\label{lge1}
|\nabla\tilde{\eta}|(x,t) \leq C t^{-\frac{1}{2}}f(Cd(x,x_0)).
\end{equation}

\noindent As in the previous section, let
$\hat{\omega}(t)=\int^{t}_{0} (d\tilde\eta(s)-\pi^*\omega)ds$, then
$\hat{\omega}(t)$
 satisfies
\begin{equation}\label{hhe2}\begin{cases}
 \left(\frac{\partial}{\partial t}-\triangle\right) \hat{\omega}(t)  =0,\\
 \hat{\omega}(0) =0.
\end{cases}
\end{equation}
The estimate (\ref{lge1})  implies $|\hat{\omega}|(x,t) \leq
Ce^{Cd^2(x,x_0)}$. As in the proof  of Theorem \ref{t2.1}, by
maximum principle (\cite[Theorem~15.2]{Li} or \cite[Lemma
2.5]{Che10}), we obtain $\hat{\omega}(t)\equiv 0$. Therefore
$\tilde\eta(T_0)$ is the desired  smooth $1$-form. The proof of
Theorem \ref{t3.1} is completed.
 \end{proof}

\vskip 1\baselineskip

\noindent \emph{The proof of Theorem \ref{t2}.} For $p\geq 1$, we
choose $f(t)=(1+t)^{p-1}$ in Theorem \ref{t3.1}. Hence, there exists
a smooth $1$-form $\eta$ on $\hat{M}$  such
    that  $\pi^{\ast}\omega=d\eta$ and
 $$|\eta|(x)\leq C(d_{\hat{M}}(x,x_0)+1)^{p-1}
 $$
for all   $x\in \hat{M}$
 where $C$ is a positive constant and  $x_0\in \hat{M}$ is a fixed
 point.
 \qed

\vskip 1\baselineskip

\noindent \emph{The proof of Theorem \ref{t3}.} By Theorem \ref{t2},
there exists  a smooth $1$-form $\eta$ on the universal covering
$\tilde{M}$ with linear growth such    that
$\pi^{\ast}\omega=d\eta$, i.e., $M$ is K\"ahler non-elliptic.  By
\cite{JZ00} and \cite{CX01},  the Euler characteristic
$(-1)^{\frac{\dim_\R M}{2}}\chi(M)\geq 0$. \qed

\vskip 2\baselineskip

 \section{Quadratic  radial isoperimetric inequality}\label{ex}

In this section, we prove that automatic groups and
$\mathrm{CAT}(0)$ groups
 satisfy the radial isoperimetric inequality
(\ref{pii}) for $p=2$ and establish Theorem \ref{t1}.

  \begin{thm} \label{aq}  Automatic groups  satisfy the  quadratic radial  isoperimetric inequality   (\ref{pii}) for $p=2$.
\end{thm}

The notion of  automatic groups  was first introduced by Epstein.
Automatic groups are finitely generated groups whose operations are
governed by automata.  Many interesting groups are automatic, e.g.
hyperbolic groups and mapping class groups (\cite{Mos95}). We only
give a brief description of the necessary notions used in the proof
of Theorem \ref{aq}, and for more details, we refer to \cite{CE92,
Ge,Oh} and the references therein.

 We shall follow the presentation
in \cite[Section~8]{Ge}. Let $S$ be a finite set and
$S^{\ast}=\{w=s_1\cdots s_n\ |\ s_i\in S\}$ be the set of  finite
sequences of letters over $S$.  A subset $L\subset S^{\ast}$ is
called a language. A finite state automaton $\mathcal{A} $ over $S$
is a $5$-tuple $\mathcal{A}=(\Gamma, i_o, S,\lambda,Y)$ satisfying

\bd \item $\Gamma$ is a finite directed graph with vertex set
$V(\Gamma)$(states) and edge set $E(\Gamma)$ (transitions);

 \item $i_o$ is a distinguished vertex of $\Gamma$ called ``initial
 state";

 \item $\lambda: E(\Gamma)\rightarrow S$ is a map labelling $E(\Gamma)$ by
 $S$;

 \item $Y\subset V(\Gamma)$ is a set of states, called ``final states".
 \ed

\noindent  For a finite state automaton $\mathcal{A}$, the language
$L(\mathcal{A})$ is the set of labels $\lambda(p)$
  of all directed paths $p$ of $\Gamma$ which begin at $i_o$ and end at any state of $Y$.
  A language $L\subset S^{\ast}$ is said to be regular if there is a finite state automaton $\mathcal{A}$ with alphabet $S$ such that $L=L(\mathcal{A})$.

  Let $G$ be a group with a finite symmetric
generator set $S=S^{-1}$. Let  $\mu: S^*\rightarrow G$ be the
evaluation map. A pair $(S,L)$ is called an automatic structure on
$G$  if $L\subset S^{\ast}$ is  a regular language satisfying \bd
 \item $\mu: L\rightarrow G$ is surjective;

 \item  there exists a number $k>0$ such  that if $w,w'\in L$ satisfy $$d_{\Gamma(G,S)}(\mu(w),\mu(w'))\leq 1,$$
 then $w,w'$ satisfy the  $k$-fellow  traveler property: \begin{equation} \label{kftp}d_{\Gamma(G,S)}(\mu(w(i)),\mu(w'(i)))\leq k \end{equation}
for all $i$ where $w(i)=s_1\cdots s_i$ is the $i$-th subword of
$w=s_1\cdots s_n$.
  \ed

  \noindent A group is called automatic if it admits an automatic structure. It is well-known that (e.g. \cite[Theorem~2.5.1]{CE92} or \cite[Section~9]{Ge}),
  for any automatic structure $(S,L)$ on $G$,
 there exists   another automatic structure $(S, L')$, where   $L'\subset L$ can be mapped bijectively onto $G$ by the evaluation map $\mu$ and $L'$ is
   also a regular language for a possibly different  finite state automaton.
   Moreover, there exists a constant $K>0$ such that
     for any $w , w'\in L'$ with $d_S(\mu(w),\mu(w'))=1$ one has \begin{equation}\label{dol}
  |L(w)-L(w')| \leq K, \end{equation}
where $L(w)$ and $L(w')$ are the word lengths.

 It is well-known that the automatic groups
satisfy the classical quadratic isoperimetric inequality (e.g.
\cite[Theorem~9.7]{Ge}).  In the following, we shall use a similar
strategy to prove that automatic groups also satisfy the stronger
radial isoperimetric inequality (\ref{pii}) for $p=2$. \vskip
1\baselineskip

 \noindent \emph{The proof of Theorem
\ref{aq}.} Let $(S,L)$ be an automatic structure on $G$ such that
the evaluation map $\mu: L\rightarrow G$ is bijective. Let
$w=s_1s_2\cdots s_n$ be  a reduced word in the free group $F_S$
representing the identity $e$ of $G$. For  $i=1,\cdots, n$,  let
$w(i)=s_1\cdots s_i$ and $\bar w(i)=\mu(w(i))$. Let $p_i$ be the
element in $L$ satisfying $\mu(p_i)=\bar w(i)$. We claim \beq L(p_i)
\leq K d_S(\bar w(i), e). \label{aaq} \eeq

\noindent  Indeed, let $a_1a_2\cdots a_{\ell(i)}$ be a reduced word
representing a minimal geodesic connecting $e$ and $\bar w(i)$ where
$a_j\in S$ and  $\ell(i)=d_S(\bar w(i), e)$. Let $p'_j\in L$ satisfy
$\mu(p'_j)=\mu(a_1\cdots a_j)$ for $j=1,\cdots, \ell(i)$. From
(\ref{dol}), we know $|L(p'_j)-L(p'_{j+1})| \leq K$. This implies
$$L(p'_{\ell({i})})\leq K\ell(i)=K d_S(\bar w(i), e).$$ Note that
$p'_{\ell(i)}=p_i$ since $\mu(p'_{\ell(i)})=\mu(p_i)=\bar w(i)$ and
$\mu$ is bijective from $L$ to $G$. Hence, we obtain (\ref{aaq}).

From the $k$-fellow traveler property (\ref{kftp}), we get
$d_S(\mu(p_{i+1}(t)), \mu(p_{i}(t))) \leq k$ for all $t$ where
$p_{i}(t)$ is the $t$-th subword of $p_i$. For fixed $i$ and $t \leq
\max \{ L(p_i), L(p_{i+1})\}$, connecting the five points
$$\mu(p_{i}(t)),\ \ \mu(p_{i}(t+1)),\ \  \mu(p_{i+1}(t+1),\ \
\mu(p_{i+1}(t),\ \ \mu(p_{i}(t))
$$ successively with $4$ minimal geodesics, we get a closed loop
$\gamma^{i}_{t}$ of length $\leq 2k+2$.  The total number of the
loops  $\gamma^{i}_{t}$  is bounded by $$ \sum_{i=1}^{L(w)}
\left(\max \{ L(p_i), L(p_{i+1})\}+1\right)\leq 2K \sum_{i=1}^{L(w)}
(d_S(\bar w(i), e)+1).$$ We  assume that the set $R$ of relators
 consists of all words in $F_S$ representing $e$ with length  $\leq 2k+2$. We obtain the radial
isoperimetric inequality (\ref{pii}) for $p=2$, i.e. \beq
\mathrm{Area}(w) \leq  2K \sum_{i=1}^{L(w)} [d_S(\bar w(i),e) +1].
\eeq The proof of Theorem \ref{aq} is complete. \qed

 \begin{thm}\label{bq}  $\mathrm{CAT}(0)$-groups satisfy  the  quadratic radial  isoperimetric inequality (\ref{pii}) for $p=2$.
 \end{thm}

 Recall that a $\mathrm{CAT}(0)$-group is a group which can act properly discontinuously and cocompactly by isometries on a proper geodesic
  $\mathrm{CAT}(0)$-metric space $(X,d)$. A geodesic metric space $(X,d)$ is called a  $\mathrm{CAT}(0)$-metric space if it satisfies the following properties.
  For any geodesic triangle $\triangle_{ABC}$ in $X$, let $\triangle_{\bar{A}\bar{B}\bar{C}}$ be the comparison triangle
   in the Euclidean plane with the same side lengths as  $\triangle_{ABC}$, i.e. $d(A,B)=d(\bar{A},\bar{B})$,
   $d(B,C)=d(\bar{B},\bar{C})$, $d(C,A)=d(\bar{C},\bar{A})$. The triangle $\triangle_{ABC}$ is at least as thin as
      $\triangle_{\bar{A}\bar{B}\bar{C}}$, i.e., for any point $D$ on the side $BC$, let $\bar D$ be the corresponding point on $\bar{BC}$,
       with $d(B,D)=d(\bar B, \bar D)$, we must have $d(A,D)\leq  d(\bar A,\bar D)$.
 Typical examples of $\mathrm{CAT}(0)$-groups are fundamental groups of compact manifolds with non-positive sectional
 curvatures.\\

 \noindent \emph{The proof of Theorem \ref{bq}.} 
 Let $(X,d)$ be the proper
  geodesic $\mathrm{CAT}(0)$-metric space on which $G$ acts properly discontinuously and cocompactly.     Fix a point $p\in X$
  and let
     $\Omega=B(p, 4D)$ be an open ball of radius $4D$ centered at $p$,  where ${D}$ is the diameter of the quotient space $X/G$.
         It is clear that for  any $q\in X$, there exists some $g\in G$ such  that  $g(q)\in B(p,2D)$.  Let $S=\{g\in G\ |\  g(\Omega)\cap \Omega \neq
         \emptyset\}$. It is obvious that $S$ is a finite symmetric subset of $G$.  Moreover, $S$ is a generating set of
         $G$. Indeed, for any $g\in G$, let $$\gamma:[0, d(p, g(p))]\>X$$
   be a minimal geodesic connecting $p$ and $g(p)$. We may assume $d(g(p),p)>2D$, otherwise $g\in
   S$. For each $i=1,2,\cdots, k=[\frac{d(g(p),p)}{D}]$, we can choose $p_i\in B(p,2D)$ and $g_i\in G$ such that $g_i(p_i)=\gamma(iD)$.
   Since $d(g^{-1}_{i}g_{i+1}(p_{i+1}), p_{i})\leq D$, we know $g^{-1}_{i}g_{i+1}(p_{i+1})\in B(p,3D)\subset \Om$ and so  $s_i:=g^{-1}_{i}g_{i+1}\in S$,
   and $s_k:=g^{-1}_kg\in S$. Then $g=g_k s_{k}=g_{k-1} s_{k-1}s_{k}=\cdots=g_1s_1s_2 \cdots
   s_{k}$. It is obvious $g_1\in S$. Hence $G$ is generated by
   $S$.

      Let $\Gamma(G,S)$ be the Cayley graph of $G$ over $S$. For any $g_1, g_2\in G$, we have
 \begin{equation}\label{qui}
  \frac{d(g_1(p), g_2(p))}{12D}\leq d_{S}(g_1, g_2)\leq \left[\frac{d(g_1(p), g_2(p))}{D}\right]+1.
 \end{equation}
 The second inequality has been proved by the above argument. The first inequality follows from the triangle inequality:
  let $g^{-1}_2g_1=s_1\cdots s_k$ with $s_i\in S$,  then
 \begin{equation}
 \begin{split}
 d( s_1\cdots s_k(p),p) & \leq  d( s_1\cdots s_k(p),s_1\cdots s_{k-1}(p))
 +\cdots+d(s_1(p),p)\\
 &=d(s_k(p),p)+\cdots+d(s_1(p),p)\\
 &\leq k\cdot 12D,
 \end{split}
 \end{equation}
where the last inequality follows from the fact that
$d(s_i(p),p)\leq 4D+4D+4D=12D$ for each $i$.
 Let $w=s_1s_2\cdots s_k\in F_S$ be a word representing the identity $e$ in $G$, $w(i)=s_1\cdots s_i$ be the $i$-th subword of $w$,
  $i=1,2,\cdots, k$ and $\bar w(i)$ the natural image in $G$. Connecting $p$ and $\bar{w}(i)(p)$
  with a minimal geodesic $$\gamma_i(s):[0, d(p, \bar{w}(i)(p)]\>X.$$

  \noindent We
 choose $\left[\frac{d(p,\bar{w}(i)p}{D}\right]+1$ points on $\gamma_i $ as above.
Let  $p^{j}_i=\gamma_i(jD)$ for $j\leq
\left[\frac{d(p,\bar{w}(i)(p))}{D}\right]$, and $p^{
\left[\frac{d(p, {w}(i)(p))}{D}\right]+1}_i=\bar{w}(i)(p) $.  Note
that
$$
d(\bar {w}(i)(p), \bar{w}(i+1)(p))=d(s_1\cdots s_i(p), s_1\cdots
s_{i+1}(p))=d(p, s_{i+1}(p))\leq 12D.
$$

Applying $\mathrm{CAT}(0)$ inequality, we obtain
\begin{equation} \label{bg} d(p^j_i, p^{j}_{i+1})\leq d(\bar{w}(i)(p),
\bar{w}(i+1)(p))  \leq 12D
\end{equation}
for all $j>0$. Actually, this can be seen from the following simple
fact: for any  triangle $\triangle_{ABC}$ on the plane
$\mathbb{R}^2$ with $|AB|\leq |AC|$, let $E$ and  $E'$ lie on  the
sides $AB$ and $AC$ respectively and $|AE|=|AE'|$, then $|EE'|\leq
|BC|$. Note that if $E=B$ and $|AB|\leq |AE'|$,   we still have
$|BE'|\leq |BC|$. If we choose $g^j_{i}\in G$ such that $d(g^j_i(p),
p^j_i) \leq 2D$,
  then by (\ref{bg}), we have
$$
d(g^j_i(p), g^{j}_{i+1}(p))\leq 16D.
$$

\noindent  By (\ref{qui}), we have   $d_{S}(g^{j}_{i},
g^j_{i+1})\leq 17$. Similarly, one has $d_{S}(g^{j}_{i},
g^{j+1}_{i})\leq 6$. On the Cayley graph $\Gamma(G,S)$,  connecting
the five points
$$g^{j}_{i},\ \  g^{j+1}_{i},\ \  g^{j+1}_{i+1},\ \  g^{j}_{i+1},\ \
g^{j}_{i}$$ successively with $4$ minimal geodesics,    we get a
closed loop $\sigma^j_{i}$ of length $\leq 46$. Let $w'$ be the word
representing the union  $\sigma^1_1\cup \sigma^2_1 \cdots $ of
closed loops, then $w=w'$ as elements in $F_S$. Define the relator
set $R=\{r\in F_S\ | \ \bar{r}=e, L(w)\leq 46 \}$, then $R$ is a
finite set. From the above argument, we obtain
$$\mathrm{Area}(w)\leq 2\sum^{L(w)}_{i=1} \left(\left[\frac{d(p,\bar{w}(i)(p))}{D}\right]+1\right)\leq
24 \sum^{L(w)}_{i=1} \left[d_S(\bar{w}(i),e)+1\right]
$$ where the second inequality follows from (\ref{qui}).   This completes the proof of Theorem \ref{bq}.
 \qed

\vskip 1\baselineskip

\noindent\emph{The proof of Theorem \ref{t1}.} If $\pi_1(G)$
 is automatic or
$\mathrm{CAT}(0)$, then by Theorem \ref{aq} and Theorem \ref{bq}, we
know $G$ satisfies the radial isoperimetric inequality (\ref{pii})
for $p=2$. Hence, by Theorem \ref{t3}, $M$ is K\"ahler non-elliptic
 and the Euler characteristic
 $(-1)^{\frac{\dim_\R M}{2}}\chi(M)\geq 0$. \qed

\end{document}